\documentclass[a4paper,11pt]{article}
\usepackage[utf8x]{inputenc}
\usepackage{amsmath}
\usepackage{amsthm}
\usepackage{amssymb}
\usepackage{amsfonts}
\usepackage[USenglish]{babel}
\usepackage{enumerate}
\usepackage[fixlanguage]{babelbib}
\usepackage{graphicx}
\usepackage{tikz}
\usepackage{wasysym}
\usepackage{hyperref}
\selectbiblanguage{UKenglish}
\usetikzlibrary{matrix, arrows}
\title{Complex analytic properties of \\ minimal Lagrangian submanifolds}
\author{Roberta Maccheroni \\ (roberta.maccheroni@unife.it)}
\date{\today}

\theoremstyle{definition}
\newtheorem{teo}{Theorem}[section]

\newtheorem{lem}[teo]{Lemma}
\newtheorem{prop}[teo]{Proposition}
\newtheorem{cor}[teo]{Corollary}

\theoremstyle{plain}
\newtheorem{defin}[teo]{Definition}

\theoremstyle{remark}

\newcommand{\R}{\mathbb{R}}
\newcommand{\C}{\mathbb{C}}
\newcommand{\D}{\mathbb{D}}

\newcommand{\p}{\mathbb{P}^3}
\newcommand{\tonde}[1]{\left( #1 \right)}
\newcommand{\quadre}[1]{\left[ #1 \right]}
\newcommand{\graffe}[1]{\left\{ #1 \right\}}
\newcommand{\modulo}[1]{\left| #1 \right|}
\newcommand{\norma}[1]{\left|\left| #1 \right|\right|}

\DeclareMathOperator{\Diff}{Diff}
\DeclareMathOperator{\Imm}{Imm}

\DeclareMathOperator{\grad}{grad}

\DeclareMathOperator{\Ric}{Ric}
\DeclareMathOperator{\cil}{cyl}

\DeclareMathOperator{\Vol}{Vol}
\DeclareMathOperator{\vol}{vol}

\DeclareMathOperator{\ind}{index}
\DeclareMathOperator{\rnk}{rank}
\DeclareMathOperator{\coker}{coker}
\begin{document}
	
	\maketitle
	
	\begin{abstract}
		
		In this article we study complex properties of minimal Lagrangian submanifolds in K\"{a}hler ambient spaces, and how they depend on the ambient curvature. In particular, we prove that, in the negative curvature case, minimal Lagrangians do not admit fillings by holomorphic discs. The proof relies on a mix of holomorphic curve techniques and on recent convexity results for a perturbed volume functional.
	\end{abstract}

\section{Introduction}
An immersed submanifold $\iota: L^n \hookrightarrow M^{2n}$ in a K\"{a}hler manifold $\tonde{M, J, g, \omega } $ is said to be \emph{minimal} if the mean curvature vector field $H$ is zero; equivalently, if it is a critical point of the Riemannian volume. The submanifold is \emph{Lagrangian} if the induced K\"{a}hler form on $L$ is zero, i.e. $\iota^* \omega \equiv 0$.
From now on we will identify $L$ with its image $\iota \tonde{L}$.

Thanks to the compatibility condition $g \tonde{ \cdot, \cdot}= \omega \tonde{ \cdot, J \cdot}$, $L$ is Lagrangian if and only if 
$$
T_pM = T_pL \stackrel{\bot}{\oplus} J \tonde{T_pL}. 
$$
It follows that a Lagrangian submanifold has a special linear-algebraic property with respect to $J$: it is \emph{totally real}, i.e. $T_pL \cap J \tonde{T_pL}= \graffe{0}.$

We are interested in the geometric properties of compact submanifolds that are simultaneously minimal and Lagrangian. 
Notice that these conditions involve only the Riemannian and symplectic ambient structures, and indeed up to now minimal Lagrangian submanifolds have been studied mainly from the Riemannian point of view, such as the second variation formula and stability under the mean curvature flow. However, the above compatibility condition suggests that these submanifolds should also have interesting complex analytic properties.

Specifically, we will investigate the existence of holomorphic discs with boundary on a fixed minimal Lagrangian, and the existence of fillings by holomorphic discs. Holomorphic discs and disc fillings are a classical problem, cf. \cite{Alexander1},\cite{DuvalSibony},\cite{gromovpseudoholom}, \cite{eliashberg}, \cite{Zehmisch}, but the literature generally focuses on symplectic or complex assumptions on the submanifold, such as the existence of complex points or being contained in the boundary of a pseudo-convex domain. By substituting those assumptions with the Riemannian condition of being minimal we are taking a novel direction, cf. also \cite{cielgold}, which we expect may have developments beyond those studied here.  We remark that the minimal Lagrangian condition is typically over-constrained, cf. \cite{bryantproceedings}, unless one restricts to K\"ahler-Einstein ambient spaces. Accordingly, this will be our main focus.

In order to put the problem into context, consider the situation in dimension 1. In this case $M$ is a Riemann surface with constant curvature and the compact minimal Lagrangians are exactly the closed geodesics. We then observe that the topological and complex analytic properties of $L$ depend on the sign of the curvature, as follows.
\begin{itemize}	
	\item
	If $M$ has positive curvature, i.e. $M$ is the sphere $S^2$, geodesics are maximal radius circles.  In this case they are homotopically trivial (see Figure \ref{fig:surfaces}), thus bound a holomorphic disc. 
	\item	 If $M$ has zero curvature, i.e. $M$ is the torus $T^2$, geodesics are not homologically trivial, so they do not bound discs.
	
	\item	  If $M$ has negative curvature, there exist examples of homologically trivial geodesics. 
	The closed geodesic on the genus $2$ surface $\Sigma_2$ in Figure \ref{fig:surfaces}, for example, is the boundary of a handle $N$. The Gauss-Bonnet theorem shows however that such curves cannot bound a disc.
	
\end{itemize}
	 \begin{figure}[h] 
	 	\centering
	 	\includegraphics[width=0.8\linewidth]{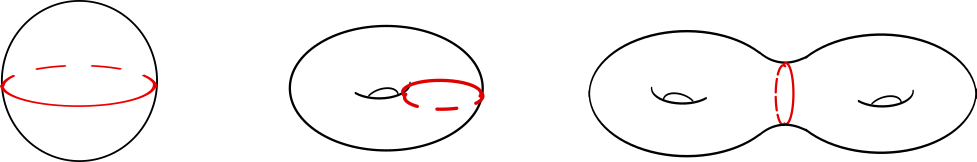}
	 	\caption{Minimal Lagrangians on surfaces}
	 	\label{fig:surfaces}
	 \end{figure}

To understand the appropriate analogue of these properties in higher dimensions, consider the K\"ahler-Einstein manifold  $\mathbb{P}^1 \tonde{\C} \times \mathbb{P}^1 \tonde{\C}$, with positive scalar curvature. The product $\gamma\times\gamma$ of maximal circles is a minimal Lagrangian $L$. Each $\gamma \times \graffe{*}$ is the boundary of a holomorphic disc; the product is homologically trivial,  and in fact $L$ admits a filling by holomorphic discs.

On the other hand, in the Ricci-flat (more precisely, Calabi-Yau) case, the theory of calibrations shows that minimal Lagrangians minimize volume in their homology class. They are thus never homologically trivial, and in particular they cannot be filled by holomorphic discs.

We are thus led to trying to understand whether the negative curvature case has a distinctive behaviour in regards to these properties, as in dimension $1$.

In Section \ref{section nonexistence of discs} 
we focus on non existence results for holomorphic discs, using standard techniques based on subharmonic functions. Thanks to these results we can examine the simplest class of examples: product spaces. In this very special situation, we show that minimal Lagrangians do not admit even a single holomorphic disc, so they clearly do not admit holomorphic disc fillings even though, in contrast with the Ricci-flat case, they may be homologically trivial.

In Section \ref{Rigidity} 
 we focus on rigidity results for holomorphic discs, using standard techniques from the theory of $J$-holomorphic curves. This leads to the following non existence result for holomorphic disc fillings, in marked contrast with the case of positive curvature seen above.

\begin{teo} \label{solid torus}
	Let $M$ be a $4$-dimensional K\"{a}hler-Einstein manifold with $\Ric \leq 0$ and let $\iota: S^1 \times S^1 \hookrightarrow M$ be a minimal Lagrangian immersed torus. Then it does not admit any filling by holomorphic discs, \emph{i.e.} there is no smooth map $F: S^1 \times \overline{\D} \longrightarrow M$ such that 
	\begin{enumerate} [(i)]
		\item $F|_{S^1 \times \partial \D} = \iota $ (up to reparametrization);
		\item $F|_{\graffe{\psi } \times \D}: \graffe{\psi } \times \D \longrightarrow M$ is a holomorphic immersion, $\forall \psi \in S^1$.
	\end{enumerate}
\end{teo}

It does not seem possible 
 to obtain an analogous result in higher dimensions using these techniques. To generalize Theorem \ref{solid torus} we thus go back to dimension $1$ to prove a new convexity result for the length functional, cf. Section \ref{length}. In order to illustrate its use, we give an alternative proof of the non existence of holomorphic discs bounded by a closed geodesic on a non positively curved Riemann surface. Our main interest in the convexity result however is that it admits a generalization to higher dimensions, explained in Section \ref{j vol}. This is based on the theory of a perturbed volume functional: the \textit{$J$-volume}, introduced in \cite{borrelli} and further studied in \cite{LP} and in \cite{LP1}. Thanks to those results, in Section \ref{disc filling} we prove our main result which roughly speaking states:
\begin{teo}
	Minimal Lagrangians in non-positive K\"{a}hler-Einstein manifolds do not admit fillings by holomorphic discs in any dimension.
\end{teo}
We refer to Theorem \ref{main thm} for the precise statement, which requires a few preliminary definitions (cf. Definition \ref{def: filling } and Definition \ref{def: rank1}). More generally, the same is true for the class of totally real \textit{$J$-minimal} submanifolds, which generalize minimal Lagrangians to ambient spaces which are not necessarily K\"ahler-Einstein.

In Section \ref{example}, following a technique used in \cite{bryant}, we exhibit two examples of totally geodesic (thus minimal) Lagrangian tori in a (non-product) 2-dimensional negative K\"{a}hler-Einstein manifold. Using the perturbation techniques of \cite{LP2} these generate many more examples of minimal Lagrangian and $J$-minimal tori (not necessarily totally geodesic) to which one can apply our results.

\medskip 
\textbf{Acknowledgements} This work is part of my PhD thesis at the University of Parma. I am grateful to my advisor Tommaso Pacini for encouraging me to study the problem of filling minimal Lagrangians and for invaluable discussions and suggestions. I would like to thank also Jonny Evans, Jason Lotay, Luciano Mari and Kai Zehmisch for useful discussions, as well as Robert Bryant and Claude LeBrun who indirectly helped me via conversations with Tommaso Pacini.
Finally I would like to thank the anonymous referee for interesting suggestions and several improvements, especially in Section \ref{section nonexistence of discs}.

\section{Non existence of holomorphic discs} \label{section nonexistence of discs}

	In this section we present some non-existence results for holomorphic discs in K\"{a}hler manifolds with non-positive sectional curvature. These techniques apply in two situations, depending on the specific boundary condition for the holomorphic disc.
	
	\subsection{Totally geodesic submanifolds}
	The first case involves totally geodesic submanifolds; from now on we will denote with $\D$ the open unit disc. 
	
	\begin{teo} \label{teo tot geodesic}
	Let $M$ be a K\"{a}hler manifold with non-positive sectional curvature and let $L$ be a totally geodesic Lagrangian submanifold. Then there is no non-constant smooth disc $u: \overline{\mathbb{D}} \rightarrow M$, holomorphic on $\mathbb{D}$ and with $u \tonde{\partial \mathbb{D}}\subseteq L$.
	\end{teo}
\begin{proof}
	Assume such $u$ does exist. Using variables $(s,t)\in\D$, set $\zeta= \partial_s u$ and $\eta= \partial_t u$. Since $M$ has an integrable structure, equation $(4.3.7)$ in \cite{mcduffsalamon} gives that
	\[
	\Delta \tonde{\modulo{du}^2}= 4 \tonde{\modulo{\nabla_s \zeta}^2+\modulo{\nabla_t \zeta}^2 - \left < R \tonde{\zeta, \eta} \eta, \zeta \right >} \geq 0,
	\]
	thus $\modulo{du}^2$ is subharmonic, because $M$ has non-positive sectional curvature.
	
	Let $\Psi(z)=\frac{z-i}{z+i}$ denote the standard biholomorphism between the upper half plane $\mathbb{H}$ and the disc $\D$, and set $v:= u \circ \Psi:  \mathbb{H} \longrightarrow M$.
The function $\modulo{dv}^2$ is again subharmonic and non-constant.

	By Lemma $4.3.1$ in \cite{mcduffsalamon} (which relies on reflection arguments for totally geodesic Lagrangian submanifolds) it holds that, if $r>0$,
	\[
	\sup_{\mathbb{H} \cap B_r} \modulo{dv}^2 \leq \frac{8}{\pi r^2} 	\cdot \int_{\mathbb{H} \cap B_{2r}} \modulo{dv}^2 dx \, dy.
	\]
Moreover, by the change of variables formula,	
	\[
	\int_{\mathbb{H}} \modulo{dv}^2 dx \, dy= \int_{\mathbb{H}} \modulo{d \tonde{u \circ \Psi}}^2 dx \, dy = \int_{\D} \modulo{du}^2 ds \, dt < + \infty,
	\] 
	because $u$ extends up to the boundary of the disc $\D$.

	We thus have that
	 $\int_{\mathbb{H} \cap B_{2r}} \modulo{dv}^2$ is uniformly bounded for any $r>0$, so 
	\[
	\sup_{\mathbb{H} } \modulo{dv}^2 \leq \lim_{r \rightarrow +\infty} \sup_{\mathbb{H} \cap B_r} \modulo{dv}^2 = 0,
	\]
	hence $dv=0$, giving a contradiction.
	
\end{proof}

\subsection{Geodesic boundary}
In the second case we impose a geodesic boundary condition.

\begin{teo} \label{teo: geodesic boundary}
Let $M$ be a K\"{a}hler manifold with non-positive sectional curvature and let $\gamma: S^1 \rightarrow M$ a closed geodesic. Then there is no non-constant smooth disc $u: \overline{\mathbb{D}} \rightarrow M$, holomorphic on $\D$, with $u \left|_{\partial \mathbb{D} }\right. = \gamma$.
\end{teo}
\begin{proof}
	Suppose by contradiction that such a disc $u: \overline{\mathbb{D}} \rightarrow M$  does exist.

	Since the function $\modulo{du}^2$ is subharmonic as seen above, harmonic function theory shows that there exists a harmonic function $h$ on $\D$ such that 
	\[
	\left \{
	\begin{array}{cc}
	 \modulo{du}^2 \leq h & \mbox{ on } \D,\\ 
	\modulo{du}^2=h & \mbox{ on } \partial \D.
	\end{array} 
	\right. 
	\]
	Let us consider the energy functional 
	\[
	E(r):= \int_{0}^{2 \pi} \left. \modulo{\frac{\partial u}{\partial \theta}}^2 \right|_{r e^{i \theta}} d \theta. 
	\]
	Since 
	\[
	\modulo{du}^2= \frac{2}{r^2} \modulo{\frac{\partial u}{\partial \theta}}^2,
	\]
	by the mean value inequality we have that
	\[
	E(r)= \frac{\, r^2}{2} \cdot \int_{0}^{2 \pi}  \modulo{du}^2\tonde{r e^{i \theta}} d \theta \leq \frac{\, r^2}{2} \cdot \int_{0}^{2 \pi}  h\tonde{r e^{i \theta}} d \theta = h(0) \cdot  \pi r^3.
	\]
	Since geodesics are critical points of the energy functional, we would have that the energy functional $E:  \quadre{0,1} \longrightarrow \R$ has a critical point in $r=1$. However, we have shown that
	 $E(r) \leq C \cdot r^3$, with equality for $r=1$: this gives a contradiction.

\end{proof}

	 \subsection{A special case: products} \label{products}
	The simplest setting in which we can study minimal Lagrangian submanifolds is that of products.
	Let $\Sigma^1, \ldots, \Sigma^n$ be Riemann surfaces endowed with K\"ahler metrics and let $\gamma_i: S^1 \rightarrow \Sigma^i$ be closed geodesics, with $i=1, \ldots, n$. Then the product manifold $M= \Sigma^1 \times  \cdots \times \Sigma^n$ inherits a K\"{a}hler structure with respect to which the submanifold $L= \gamma_1 \times \cdots \times \gamma_n$ is minimal Lagrangian. The ambient structure is K\"ahler--Einstein if the curvature is uniformly constant, for all surfaces. For example, let $\Sigma_2$ be the genus $2$ surface with its standard hyperbolic structure and let $\gamma: S^1 \rightarrow \Sigma_2$ be the closed geodesic of Figure \ref{fig:surfaces}.
	Then the submanifold $L:= \gamma \times \gamma: S^1 \times S^1 \longrightarrow M:= \Sigma_2 \times \Sigma_2$ is a homologically trivial minimal Lagrangian torus, since it is the boundary of the $3$-manifold $N \times \gamma $, where $N$ is the handle bounded by $\gamma$.
	
	We want to show that in this case such minimal Lagrangians do not bound holomorphic discs.

	\begin{cor} \label{topological filling}
		Let $\Sigma^1, \ldots, \Sigma^n$ be Riemann surfaces with nonpositive scalar curvature and let $\gamma_j: S^1 \rightarrow \Sigma^j$ be  closed geodesics, for $j=1, \ldots, n$. Denoting with $M$ the product space $\Sigma^1 \times \ldots \times \Sigma^n$ and with $L$ the submanifold $\gamma_1 \times \cdots \times \gamma_n$, there does not exist any nonconstant continuous map $f: \overline{\D} \longrightarrow M $, holomorphic on $\D$, such that $f \tonde{\partial \D} \subseteq L$. 
	\end{cor}

Since $\gamma_1 \times \cdots \times \gamma_n$ is a totally geodesic submanifold in $M$, Corollary \ref{topological filling} is a particular case of Theorem \ref{teo tot geodesic}.

	 \section{Rigidity of holomorphic discs in dimension 2} \label{Rigidity}
We now want to show that, in complex dimension 2, minimal Lagrangians never admit holomorphic fillings. In this dimension we can achieve this using standard results from the theory of $J$-holomorphic curves.

	Let us consider an almost-complex manifold $M$, a totally real submanifold $L$ and an immersed holomorphic disc $u: \D \rightarrow M$ with $u \tonde{\partial \D} \subseteq L$. In this section we will assume that all the holomorphic discs are smooth up to the boundary. 
	
	Let us identify $\D $ with its image $ u \tonde{\D}$ and consider the normal bundle $V:=u^* \tonde{TM / T \D}$ over $\D$ and the totally real subbundle $V_0:= u^* \tonde{TL / T \partial \D}$ over $\partial \D$. We will denote by $W^{k,q}_{V_0} \tonde{\D, V }$ the closure of the space of vector fields
	\[
	\Lambda^0_{V_0} \tonde{\D, V }:= \graffe{ \xi \in \Lambda^0 \tonde{\D, V } \, | \, \xi \tonde{\partial \D} \subseteq V_0}
	\] 
	in the Sobolev space $W^{k,q} \tonde{\D, V }$. Let $W^{k,q}_{V_0} \tonde{\D, V }_{\varepsilon} \subseteq W^{k,q}_{V_0} \tonde{\D, V }$ denote the ball of radius $\varepsilon$ centred at the zero section.

	As in \cite{LeBrun}, there exists a differentiable map
	\[
	\mathcal{D}: W^{k,q}_{V_0} \tonde{\D, V }_{\varepsilon} \longrightarrow W^{k-1,q} \tonde{\Lambda^{0,1}\tonde{\D}, V },
	\]
	such that $\mathcal{D}^{-1} (0)$ consists of all holomorphic curves $\tonde{\D, \partial \D} \rightarrow \tonde{M, L}$ sufficently near $u$ and whose linearization $\mathcal{D}_0$ at $0$ is exactly the canonical operator
	\[
	\overline{\partial}: W^{k,q}_{V_0} \tonde{\D, V } \longrightarrow W^{k-1,q} \tonde{\Lambda^{0,1}\tonde{\D}, V }.
	\] 
	Thanks to this fact, it is possible to prove the rigidity of holomorphic discs in dimension $2$.
	
	\begin{prop} \label{prop rigidity}
		Let $L$ be a minimal Lagrangian submanifold in a nonpositive K\"{a}hler-Einstein manifold $M$ of complex dimension $2$. If $u: \D \rightarrow M$ is a holomorphic immersed disc, and $u \tonde{\partial \D} \subseteq L$, then $u$ is rigid, i.e. there is no transversal $1$-parameter family of holomorphic discs starting from $u$. 
	\end{prop}
	\begin{proof}
Using the notation above, according to the axiomatic definition of the Maslov index in \cite{mcduffsalamon}, Appendix C,
 			\[
			\mu \tonde{V, V_0}= \mu \tonde{u^*TM, u^*TL} - \mu \tonde{T \D, T \partial \D},
 			\]
 			and $\mu \tonde{T \D, T \partial \D}=2$. 
Furthermore, \cite{cielgold} proves that each minimal Lagrangian submanifold $L$ in a K\"{a}hler-Einstein ambient $M$ is monotone; specifically, 
 			\[
 			\mu \tonde{u^*TM,u^* TL}= 2\lambda \int_{\D} \omega,
 			\]
 			where $\lambda $ is the K\"{a}hler-Einstein constant. Thus, in our case, $\mu \tonde{u^*TM,u^* TL}$ is nonpositive, so $ \mu \tonde{V,V_0} \leq -2$.
 			
 			Since the bundle $V$ has complex rank $1$ and $ \mu \tonde{V, V_0}$ is negative, we can apply Theorem C.1.10, (iii) in \cite{mcduffsalamon}, thus the operator $\mathcal{D}_0$ is injective. Since we are in dimension $4$ and
 				\[
 			\ker\mathcal{D}_0= \graffe{ \mbox{ infinitesimal transversal holomorphic deformations of } u},
 			\]
 			 the holomorphic disc does not admit any transversal holomorphic deformation.

	\end{proof}
	In particular, Theorem \ref{solid torus} follows.
	
	In arbitrary dimensions, the theory of $J$-holomorphic curves allows us to prove the same result only in the case of regular discs, i.e. discs for which $\mathcal{D}_0$ is surjective.
	
	\begin{prop}
		Let $L \subset M$ be a minimal Lagrangian submanifold in a nonpositive K\"{a}hler-Einstein ambient of complex dimension $n$. Then $L$ does not admit a filling by regular immersed holomorphic discs. 
	\end{prop}

	\begin{proof}
		If a disc $u: \D \rightarrow M$ is regular, it generates a smooth moduli space $\mathcal{M}$ of real dimension equal to \[ \dim \ker \mathcal{D}_0 = \dim \ker \mathcal{D}_0- \dim \coker \mathcal{D}_0= \ind \tonde{\mathcal{D}_0}. \]
		
		Thanks to the Riemann-Roch theorem, as stated in \cite{mcduffsalamon}, Theorem C.1.10, the linearized operator $\mathcal{D}_0$ is Fredholm, and
		\[
		\ind \tonde{\mathcal{D}_0}= \rnk \tonde{V} \cdot \chi \tonde{\D}+ \mu \tonde{V, V_0} = (n-1) \cdot \chi \tonde{\D} +\mu \tonde{V, V_0} ,
		\]
		thus we have
		\begin{eqnarray}
		\dim \ker \mathcal{D}_0 & = & (n-1) \cdot \chi \tonde{\D}+ \mu \tonde{V, V_0} = \nonumber   \\ &=& n-1 +\mu \tonde{u^*TM, u^*TL } - \mu \tonde{T \D, T \partial \D} = \nonumber \\ &=& n+ \mu \tonde{u^*TM, u^*TL } -3. \nonumber
		\end{eqnarray}
		Since the dimension of an analytic filling is $n+1$ and the dimension of a disc is $2$, it must be $
		\dim \mathcal{M} \geq n-1 $, so
		\[
		n+ \mu \tonde{u^*TM, u^*TL } -3 \geq n-1  \; \Rightarrow  \mu\tonde{u^*TM, u^*TL } \geq 2, 
		\]
		which contradicts the fact that $L$ is nonpositive monotone, as proved  in \cite{cielgold}.
		
	\end{proof}
	
However, in some cases one can prove that holomorphic discs are certainly not regular, because the index is negative. For example, if $M$ is negative K\"{a}hler-Einstein and
	\begin{itemize}
		\item if $n=3$, then $ n+ \mu\tonde{u^*TM, u^*TL }-3 =\mu \tonde{u^*TM, u^*TL }< 0$, because $L$ is monotone;
		\item if $n=4$, the $n+ \mu \tonde{u^*TM, u^*TL } -3 =\mu \tonde{u^*TM, u^*TL } +1 \leq -1$, because $\mu \tonde{u^*TM, u^*TL } \leq -2$, since it is even and negative.
	\end{itemize}
	Thus in dimension $3$ and $4$ we can not use the theory of $J$-holomorphic curves to prove general results regarding the non existence of holomorphic disc fillings.

	\section{Convexity of the length functional} \label{length}
	In order to generalize our non-filling results to higher dimensions, the above standard ideas are not sufficient. 
	The purpose of this section is to introduce a new technique and to illustrate its use in the simplest situation. Specifically, we prove a convexity result for the Riemannian volume functional in dimension $1$, i.e. the length functional, and use it to give an alternative proof of Theorem \ref{teo: geodesic boundary}. This is a concrete example of the more general results, in higher dimensions, discussed in Sections \ref{j vol} and \ref{disc filling}.

	We will need the following classical \emph{isothermal coordinates} theorem which shows that a Riemannian structure on an oriented surface induces a complex structure, and that the metric is locally conformal to the Euclidean one. See \cite{Jost} for the proof.
	\begin{teo} \label{Jost}
		Let $\tonde{\Sigma, g}$ be an oriented Riemannian surface. Then $\Sigma$ can be made into a Riemann surface, i.e. it admits a complex structure compatible with $g$. Local holomorphic coordinates are given by smoothly invertible solutions of the differential equation 
		\[
		w_{\overline{z}}= \mu w_z,
		\]
		where $\mu$ is a complex valued function which depends on the metric $g$.
		
		In such coordinates the metric has the form $e^{\sigma} d w \otimes d \overline{w}$ and the scalar curvature is $K=-\frac{1}{e^{\sigma}} \, \Delta \sigma$, where $\Delta $ is the Laplacian operator.
	\end{teo}
	Thanks to the complex structure $J$ compatible with the metric $g$, it is possible to define the symplectic form
	\[
	\omega \tonde{\cdot, \cdot}= g \tonde{J \cdot, \cdot},
	\]
	and the hermitian metric
	\[
	h:= g - i \omega.
	\]
	From now on, $J$ will denote this complex structure on $\tonde{\Sigma, g}$.

	\begin{teo} \label{length functional}
		Let $\tonde{\Sigma, g}$ be a Riemannian surface with nonpositive curvature and let $u =u(z)=u(r, \theta): \overline{\D} \rightarrow \tonde{\Sigma, J}$ be a smooth map, holomorphic on $\D$. Then the length functional 
		\[
		L: (0,1] \longrightarrow \R, \qquad r \mapsto \int_0^{2\pi}   \norma{d u |_{(r,\cdot)}\tonde{\partial \theta}}_g d \theta
		\]
		is nondecreasing and convex with respect to the variable $\log r$.
	\end{teo}
	\begin{proof}
		Denote with $\alpha_r$ the circle $\graffe{r e^{i \theta} } \subseteq \overline{\D}$ and consider the function
		\begin{eqnarray}
		f(z)= f(r e^{i \theta}) & := &  \norma{d u |_{(r,\cdot)}\tonde{\partial \theta}}_g   = \norma{du \tonde{\partial z} }_h   \cdot \left| \frac{\partial z}{\partial \theta}\right|= \nonumber \\ &= & \norma{du \tonde{\partial z} }_h  \cdot r =  \norma{du \tonde{\partial z} \cdot z}_h, \nonumber 
		\end{eqnarray}
		where $h$ is the hermitian metric induced by $g$.
		
		We want to show that $f: \overline{\D} \longrightarrow \R$ is subharmonic.	Thanks to Theorem \ref{Jost}, for each point of $\Imm \tonde{f}$ we can consider a holomorphic chart $\varphi$ defined on a neighborhood of the point.  
		Thus we have that 
		\[	
		\norma{du \tonde{\partial z} \cdot z}_h= e^{\sigma \circ \varphi \circ u (z)} \cdot  \left| \frac{\partial \tonde{\varphi \circ u}}{\partial z} \cdot z \right|_{std} .
		\]

		Since $K \leq 0$, we have that $ \Delta \sigma \geq 0$.
		
		The function $\sigma $ is subharmonic; since $\varphi $ and $f$ are holomorphic it follows that $\sigma \circ \varphi \circ u$ is subharmonic too. 
		Furthermore $\frac{\partial \tonde{\varphi \circ u}}{\partial z}$ is holomorphic, so
		\[
		\tonde{\sigma \circ \varphi \circ u}(z)+ \log \left| \frac{\partial \tonde{\varphi \circ u}}{\partial z} \cdot z \right| 
		\]
		is locally subharmonic, so 
		\[
		f(z)= e^{\sigma \circ \varphi \circ u(z)} \cdot  \left| \frac{\partial \tonde{\varphi \circ u}}{\partial z} \cdot z \right|
		 \] is locally subharmonic too. Covering $\Imm(f)$ with local holomorphic charts, we have that $f(z)$ is subharmonic.
		\\A standard argument shows that $L$ is nondecreasing: choose $0 < r_1 <r_2 <1$ and let $k(z)$ be the function harmonic in $|z|< r_2$, continuous in $|z| \leq r_2$ and equal to $u(z)$ on $\alpha_{r_2}$. Then $u(z) \leq k(z)$ in $|z| \leq r_2$, so
		\begin{eqnarray} 
		L(r_1) & \leq & \int_0^{2\pi}  k \tonde{r_1 e^{i \theta}} d \theta = 2 \pi  k(0) =  \nonumber \\ \nonumber 
		& = & \int_0^{2\pi}  k \tonde{r_2 e^{i \theta}} d \theta =\int_0^{2\pi}  u\tonde{r_2 e^{i \theta}} d \theta = L(r_2).
		\end{eqnarray}
		Furthermore, the functional $L$ is convex with respect to the variable $\log r$: let $k(z)$ be the harmonic function in the annulus $A= \graffe{r_1< |z|<r_2}$, continuous in $r_1 \leq |z| \leq r_2$ and equal to $f(z)$ on the two boundary components of $A$. Then $f(z) \leq k(z)$ for $r_1 \leq |z| \leq r_2$ and
		\begin{eqnarray}
		\frac{d }{d r} \int_0^{2 \pi} \ k\tonde{r e^{i \theta}} d \theta = \int_0^{2 \pi} \frac{d }{d r} k \tonde{r e^{i \theta}} d \theta = \frac{1}{r}\int_0^{2\pi}  \frac{\partial k}{\partial \mathbf{n}} d \sigma, \nonumber 
		\end{eqnarray}
		where $d \sigma= r d \theta$ is the element of arc length and $\frac{\partial k }{\partial \mathbf{n}}$ denotes the normal derivative, \emph{i.e.} $<\grad k, \mathbf{n} >$, where $\mathbf{n}$ is the normal vector.
		
		As $k$ is a harmonic map and thanks to the Divergence Theorem we have that:
		\begin{eqnarray}
		0 & = & \int_A  \Delta k= \int_{\partial A} <\grad k, \mathbf{n}_{ext} > d \sigma= \nonumber \\
		& = & \int_{B_{r_2}} <\grad k, \mathbf{n} > d \sigma - \int_{B_{r_1}} <\grad k, \mathbf{n} > d \sigma, \nonumber
		\end{eqnarray}
		so $\int_0^{2\pi}  \frac{\partial k}{\partial \mathbf{n}} d \sigma$ does not depend on $r$, and	we can conclude that, for $r \in A$,
		\[
		L(r) \leq \int_0^{2 \pi} \ k \tonde{r e^{i \theta}} d \theta= a\log r +C.
		\]
		
		Furthermore, $f(z)$ and $k(z)$ coincide on $\partial A$, so the functional $L$ is convex with respect to the variable $\log r$.
		
	\end{proof}

This result gives an alternative proof of the non existence of holomorphic discs with geodesic boundary: 

	\begin{proof}[Second proof of Theorem \ref{teo: geodesic boundary}]
		If a smooth map
		\[
		u=u(z):\overline{\D} \longrightarrow \Sigma,
		\]
		holomorphic on $\D$	existed, $r=1$ would be a critical point for the length functional $L: (0,1] \rightarrow \R$ defined above, since $u|_{\alpha_1}\equiv \gamma$ is a geodesic. Thanks to Theorem \ref{length functional} the functional 
		\[
		L \circ \exp: (- \infty, 0] \longrightarrow  \R
		\] 
		has a critical point in $t=0$, it is convex and tends to $0$ when $t$ tends to $- \infty $: this  gives a contradiction.
		
	\end{proof}

	\section{The $J$-Volume functional} \label{j vol}
	In this section we introduce an alternative volume functional defined in \cite{borrelli} and we state its properties, proved by J. Lotay and T. Pacini in \cite{LP} and \cite{LP1}, which generalize the convexity of the length functional proved in Section \ref{length}.
	
	\begin{defin}
		Let $\tonde{M,  J, h}$ be an almost Hermitian manifold. Denoting with $g$ the real part of the hermitian metric of $M$, and with $TR^+$ the Grassmannian of oriented totally real $n$-planes, we define the function
		\[
		\rho_J: TR^+ \rightarrow \R, \; \;\;\; \pi \longmapsto \sqrt{\vol_g\tonde{e_1, \ldots, e_n, J e_1, \ldots, Je_n}},
		\]
		where $e_1, \ldots, e_n $ is a positive orthonormal basis for the totally real $n$-plane $\pi$.
		
		Fix a totally real submanifold $\iota: L \rightarrow  M$. We can define the $J$-volume form as
		\[
		\vol_J = \rho_J \vol_g,
		\] 
		where $\vol_g$ is the standard Riemannian volume on $L$.
	\end{defin}
	
	Observe that $\rho_J \tonde{\pi} \leq 1$ and that the equality holds if and only if $\pi$ is Lagrangian. 
	
	We can extend the map $\rho_J$ to the Grassmannian of $n$-planes, setting $\rho_J (\pi)=0$ and $\sigma \quadre{\pi}=0  $, when $\pi$ contains a complex line. 

	We will say that a totally real submanifold is \emph{$J$-minimal } if it is a critical point for the $J$-volume functional. See \cite{borrelli}, \cite{LP2} for examples of $J$-minimal submanifolds. The following results hold (cf. \cite{LP}, \cite{LP1}):
	\begin{lem} \label{minimi}
		For any compact oriented $n$-dimensional submanifold $L$ in an almost Hermitian manifold $\tonde{M, J, h}$, we have $\Vol_J (L) \leq \Vol_g (L) $ and the equality holds if and only if $L$ is Lagrangian. In particular, the values of $\Vol_J$ and $\Vol_g$ and their first derivatives coincide on Lagrangian submanifolds.
	\end{lem}  
	\begin{proof}
		The first part follows from the fact that $\rho_J \leq 1$ with equality if and only if $L$ is Lagrangian. 
		
		To prove the second statement, let $\graffe{L_t}$ be a one-parameter family of totally real submanifolds such that $L_0$ is Lagrangian. Consider the real functions $f(t):= \Vol_g \tonde{L_t}$ and $g(t):= \Vol_J \tonde{L_t}$; we have that $g-f \geq 0$. Since it is null for $t=0$, this is minimum, hence it is a critical point, so $f^{\prime}(0)= g^{\prime}(0)=0$.

	\end{proof}

	\begin{teo} \label{j minimal}
		Let M be a K\"{a}hler–Einstein manifold with $\Ric \neq 0$. Then
		the set of $J$-minimal submanifolds coincides with the set of minimal Lagrangians.
	\end{teo}

	Interesting properties hold for the $J$-volume functional adding assumptions on the ambient manifold. 
	
	\begin{defin} \label{special curve} Let $\mathcal{T}$ be the space of totally real immersions of $L$ into $M$ which are homotopic, through totally real immersions, to the given $\iota$, up to orientation preserving diffeomorphisms of $L$.
		
		A one parameter family  $\graffe{\iota_t: L \rightarrow M}_{t \in \R}$ in $\mathcal{T}$ is a \emph{geodesic in $\mathcal{T}$} if and only if there exists a fixed vector field $X \in \Lambda^0 \tonde{TL} $ such that
		\begin{equation} \label{geodetiche}
		\frac{\partial \iota_t}{\partial t}= J \tonde{d \iota_t \tonde{X}}.
		\end{equation}

		We say that a functional $f: \mathcal{T} \longrightarrow \R$ is convex (respectively, strictly convex) if and only if it restricts to a convex function (respectively, strictly convex function) in one variable along any geodesic in $ \mathcal{T}$.
	\end{defin}
	
	Such a family is called geodesic in $\mathcal{T}$ because there exists a connection on $\mathcal{T}$ with respect to which it is a geodesic; cf \cite{LP}, Section 2.
	
	In \cite{LP} the following result about the convexity of the $J$-volume functional is proved.
	\begin{teo} \label{convexity}
		Let $M$ be a K\"{a}hler manifold with $\Ric \leq 0$ (respectively, $\Ric < 0$). Then the $J $-volume functional is convex (respectively, strictly convex). 
	\end{teo}
Observe that in dimension $1$ any totally real curve is Lagrangian, so the $J $-volume functional, the Riemannian volume functional and the length functional coincide. In this case the statement above generalizes Theorem \ref{length functional}.

	\section{Non existence of a holomorphic discs filling} \label{disc filling}
	In Section \ref{Rigidity} we proved the non existence of a filling by holomorphic discs for minimal Lagrangian tori, as stated in Theorem \ref{solid torus}. 
	
	In this section we will generalize that result to higher dimensions.

	Let $L$ a compact, oriented $n$-dimensional manifold such that $L$ admits a locally trivial $S^1$-fibre bundle structure, \emph{i.e.} there exists a $S^1$-fibre bundle $\pi: L \rightarrow B$. Observe that $\pi: L \rightarrow B$ is a restriction of a complex fiber line bundle $\tonde{N, h} \rightarrow B$, the restrictions to the circles of radius $r$ on each fiber are circle bundles themselves and that the restriction $N \rightarrow B$ to the unit discs is a disc bundle. 
	
	For example, the $n-$dimensional torus $\tonde{S^1}^n$, seen as the trivial $S^1$-bundle over the $(n-1)$-torus, is the restriction of the complex trivial line bundle $\C \times \tonde{S^1}^{n-1} \rightarrow \tonde{S^1}^{n-1}$. 
	Also the Hopf fibration is the restriction to the sphere $S^{2n+1}$ of the tautological line bundle over $\mathbb{P}^n \tonde{\C}$.
	
	In this setting, we want to generalize the concept of \emph{filling by holomorphic discs}, used in Theorem \ref{solid torus}.
\begin{defin} \label{def: filling }
Using the notation above, let $L \hookrightarrow M $ be an immersion which extends to an immersion $\iota^{\prime}: N \rightarrow M$. If its restriction to each fiber $\D_p= \pi^{-1}(p)$ is holomorphic, we will say that it is a \emph{filling by holomorphic discs}.  
\end{defin}

	\begin{defin} \label{def: rank1}
		Let $M$ be a $2n -$dimensional K\"{a}hler manifold; an immersed submanifold $S$ has \emph{complex rank 1} if, for each $p \in S$, the tangent space $T_pS$ contains exactly one complex line. 
	\end{defin}

We can now prove our main result, which generalizes Theorem \ref{solid torus} to higher dimensions, since a solid torus in a $4$-dimensional K\"{a}hler-Einstein manifold has complex rank $1$. 

	\begin{teo} \label{main thm}
		Let $M$ be a $2n$-dimensional K\"{a}hler manifold with $\Ric \leq 0$ and let $L$ be an oriented compact $n$-manifold which admits a locally trivial $S^1$-fibre bundle structure. If $\iota: L \hookrightarrow M$ is a $J$-minimal immersion, it does not admit any complex rank 1 filling by holomorphic discs.
	\end{teo}

	\begin{proof}
		The punctured disc $\overline{\D} \setminus \graffe{0}$ (with the standard complex structure as a subset of the complex plane $\C$) is biholomorphic to the cylinder $S^1 \times \R^+$, endowed with the complex structure given by 
		\[
		J_{\cil} \tonde{\partial \theta} = \partial t, \qquad  J_{\cil} \tonde{\partial t }= -\partial \theta,\]
		where $\graffe{\partial \theta, \partial t}$
		is the standard oriented basis. The biholomorphism is given by
		\[
		\varphi: S^1 \times \R^+ \longrightarrow \overline{\D} \setminus \graffe{0}; \quad \tonde{\theta, t} \mapsto e^{i \theta} \cdot e^{-t}= \tonde{e^{-t} \cos \theta ,e^{-t} \sin \theta },
		\]
		and, in particular, $\varphi \tonde{S^1 \times \graffe{0}}= \partial \D$.
		
		Suppose by contradiction that such a complex rank $1$ filling by holomorphic discs $\iota^{\prime}: N \rightarrow M$ existed; thus the only complex line in each tangent space $T_pN$ would be $\graffe{\partial \theta, \partial t}$, tangent to the holomorphic disc. 
		
		We would have a family $\graffe{\iota_t: L \rightarrow M}$, with $t \in [0, +\infty)$, of immersed submanifolds given by the restriction of $\iota^{\prime}: N \rightarrow B$ to the circle bundle of radius $r=e^{-t}$: observe that $\iota_0 $ coincides with the $J$-minimal $ \iota: L \rightarrow M$. 
		
		 For each $t \in [ 0, + \infty )$ we have that $\iota_t: L \rightarrow M$ is a totally real submanifold: if a tangent space $T_pL$ contained a complex line, it would contradict that $\iota^{\prime}: N \rightarrow M$ has complex rank $1$. 
		
		We claim that $\graffe{\iota_t}$ is a geodesic in $\mathcal{T}$ (see Definition \ref{special curve}): considering the vector field $X= \partial \theta$ we have that
		\[
		d \iota_t \tonde{X}= \frac{\partial \iota_t}{\partial \theta}= \frac{\partial \iota^{\prime}}{\partial \theta}, \qquad \mbox{ and } \qquad \frac{\partial \iota_t}{\partial t}=  \frac{\partial \iota^{\prime}}{\partial t},
		\]
		so equality (\ref{geodetiche}) holds because it is equivalent to
		\[
		\frac{\partial \iota^{\prime}}{\partial t}= J \tonde{\frac{\partial \iota^{\prime}}{\partial \theta}},
		\]
	 which is the Cauchy-Riemann equation for each holomorphic disc.  
		
		Applying Theorem \ref{convexity} we have that the $J$-volume functional restricted to the curve $\graffe{\iota_t} \subset \mathcal{T}$ is convex because $\Ric \leq 0$. Thanks to Lemma \ref{minimi} we have that $t=0$ is a critical point because $\iota_0$ is $J$-minimal, and $ \Vol_J(\iota_0)>0$. Furthermore it tends to $0$ for $t$ tending to $+ \infty$ because the Riemannian volume of $\iota_t \tonde{ L}$ tends to $0$ and $ \Vol_J \leq \Vol$ (see Lemma \ref{minimi}), and it leads to a contradiction because the $J$-volume functional is convex.

	\end{proof}
Observe that, according to Theorem \ref{j minimal}, the main theorem concerns exactly minimal Lagrangian submanifolds when the ambient space is negative K\"{a}hler-Einstein.

In Calabi-Yau manifolds the statement above is trivial for special Lagrangian and $J$-minimal totally real submanifolds since they are calibrated submanifolds, thus their homology classes are not null.
\medskip

	 \section{Examples} \label{example}

We want to look for nontrivial examples of minimal Lagrangians in a K\"{a}hler manifold $\tonde{M, J, g, \omega } $. 

Denoting by $\rho$ the Ricci $2$-form, the Lagrangian condition implies that
\[
d \quadre{\omega \tonde{H, \cdot}}= \rho.
\] 
Thus a minimal Lagrangian submanifold $L$ would have $\rho|_L =0$, and looking for examples of minimal Lagrangian submanifolds in K\"{a}hler ambients is an overconstrained problem. See \cite{bryantproceedings} for details.

This last condition is automatically verified on a Lagrangian submanifold $L$ if $M$ is a K\"{a}hler-Einstein manifold, \emph{i.e.} if $ \rho \tonde{\cdot, \cdot}= \lambda \omega \tonde{\cdot, \cdot}$, for some $\lambda \in \R$.

Thus K\"{a}hler-Einstein manifolds are the suitable context in which look for minimal Lagrangian submanifolds. Very few examples are known; see \cite{YngIngLee} for results about existence of minimal Lagrangian tori in product spaces.  
We will exhibit a few more, following \cite{bryant}.

Let us consider a homogeneous polynomial $p$ of degree $d$ on $\C^{n+1}$, with $d > n+1$, and suppose that the zero locus 
\[
M= \graffe{ \quadre{z} \in \mathbb{P}^n \tonde{\C} \, | \, p \tonde{\quadre{z}}=0 }
\]
is smooth. Denoting with $J$ the induced complex structure, by the adjunction formula we have that its first Chern class is 
\[
c_1(M, J)= \tonde{n+1-d} \quadre{\omega},
\] 
where $\omega$ is the pull back of the Fubini-Study metric on $\mathbb{P}^n \tonde{\C}$. Observe that $c_1(M, J)<0$. 
A theorem of Aubin and Yau (see \cite{aubin}, \cite{yau}) guarantees that there is a unique K\"{a}hler-Einstein $2$-form $\overline{\omega}$ on $M$ in the class $-c_1(M,J)$. Let $g$ denote the corresponding metric.

If $p$ has real coefficients, $M$ is invariant under complex conjugation on $\mathbb{P}^n \tonde{\C}$: its restriction to $M$ gives a \emph{real structure} $c \in \Diff \tonde{M}$, i.e. an antiholomorphic involution such that $c^*J=-J$.  

Since both $c^* \overline{\omega}$ and $-\overline{\omega}$ are K\"{a}hler-Einstein for $\tonde{M, c^*J}$, it follows that $c^* \overline{\omega}=-\overline{\omega}$ by uniqueness, and that $c$ is an isometry with respect to the metric $g$:
\[
c^*g \tonde{\cdot, \cdot}= g \tonde{c \cdot, c \cdot}= \overline{\omega} \tonde{-J (c \cdot), c \cdot}= \overline{\omega} \tonde{J \cdot,  \cdot}= g \tonde{\cdot, \cdot}.
\] 
The fixed point set $L$ of $c$, called the \emph{real locus}, if non empty, is a totally geodesic (thus minimal) Lagrangian submanifold of $M$, in fact
\[
\overline{\omega}_p \tonde{v,w}=\overline{\omega}_p \tonde{c(v), c(w)}= c^* \overline{\omega}_p \tonde{v,w}= - \overline{\omega}_p \tonde{v,w},
\]
for any $p\in L$ and $v,w \in T_pL$, thus $\overline{\omega}|_L \equiv 0$.

\paragraph{First example:} Let us consider the homogeneous polynomial 
\[
P \tonde{X,Y,Z,W}=X^6+Y^6-Z^6-W^6.
 \]
  The zero locus 
\[
M= \graffe{\quadre{X:Y:Z:W} \in \p\tonde{\C} \, | \, X^6+Y^6-Z^6-W^6=0}
\]
admits a negative K\"{a}hler-Einstein structure, and the real locus
\[
T:= M \cap \p\tonde{\R} = \graffe{\quadre{x:y:z:w} \in \p\tonde{\R} \, | \, x^6+y^6-z^6-w^6=0}
\]
is a minimal Lagrangian submanifold of real dimension $2$, as above.

Observe that the curve $\mathcal{C}= \graffe{\tonde{x,y} \in \R^2 \, | \, x^6+y^6=1}$ is diffeomorphic to a circle, so 
\[
T= M \cap \p\tonde{\R} \cong \tonde{\mathcal{C} \times \mathcal{C}  } / \sim \, , 
\]
where $\tonde{\tonde{x,y}, \tonde{z,w}} \sim \tonde{\tonde{-x,-y}, \tonde{-z,-w}}\ $; hence $T$ is diffeomorphic to a $2$-torus.
Thus $T$ is a minimal Lagrangian torus in a negative K\"{a}hler-Einstein manifold.

\paragraph{Second example:} Let us consider the polynomial 
\[
q(x,y,z)= \tonde{x^2+ y^2-1}^4 +z^8.
\]
The critical locus of $q$ in $\R^3$ is the circle 
\[
C= \graffe{ \tonde{x,y,z} \in \R^3 \, | \, x^2+ y^2-1=0,\, z=0 }
\] plus the origin $O= \tonde{0,0,0}$. Since $q(C)=0$ and $q(O)=1$, the regular values of $q$ are all the real numbers except $0$ and $1$. 

Let $\varepsilon $ be a real number in $\tonde{0, 1}$ and consider the polynomial $q_{\varepsilon}:= q - \varepsilon$. Since $0$ is not a critical value for $q_{\varepsilon}$, the zero locus of $q_{\varepsilon}$ is smooth, and is a boundary component of the region $q^{-1}\quadre{0, \varepsilon}$, which retracts onto $C$. 

Now let us consider the homogeneous polynomial
\[
Q\tonde{X,Y,Z,W}= (X^2+Y^2-W^2)^4 +Z^8- \varepsilon W^8= W^8 \cdot q_{\varepsilon} \tonde{\frac{X}{W} \, , \frac{Y}{W}\, ,\frac{Z}{W}}.
\]
Its singular points in $\p\tonde{\C}$ are the two nonreal points $  \quadre{\pm i:1:0:0}$, hence $M= \graffe{Q=0} \subset \p\tonde{\C}$ is a compact K\"{a}hler orbifold, so the Calabi conjecture holds; cf. \cite{joyce}. Thus, away from the singular points, $M$ admits a negative K\"{a}hler-Einstein structure.

 Since $Q(X,Y,Z,W)=W=0$ only at the origin, the real locus is smooth and is given by
\[
T:= M \cap \p \tonde{\R}= \graffe{ \quadre{X:Y:Z:1} \in \p \tonde{\R} \, | \, q_{\varepsilon} \tonde{X,Y,Z}=0 }.
\]
It is a minimal Lagrangian torus as above.

\paragraph{Conclusions}

Non existence results in Section \ref{section nonexistence of discs} can be applied to very special examples: totally geodesic Lagrangian submanifolds in K\"{a}hler manifolds with non-positive sectional curvature.\\
It is possible to apply Proposition \ref{prop rigidity} to the two examples above; in particular they do not admit any filling by holomorphic discs.
\\Moreover, thanks to Corollary 6.1 in \cite{LP2}, any small perturbation of the polynomial $p$ gives a new K\"{a}hler-Einstein manifold and a locally unique minimal Lagrangian submanifold (in general, not totally geodesic). Proposition \ref{prop rigidity} can be applied also to these examples to prove that they do not admit any filling by holomorphic discs.

Analogous examples of minimal Lagrangians in higher dimensions can be studied only via Theorem \ref{main thm}. Furthermore, any small perturbation of $M$ in the same K\"{a}hler class, with negative Ricci curvature, has a locally unique $J$-minimal compact submanifold (see Corollary 6.1 in \cite{LP2}). 
Such submanifolds are not minimal Lagrangian in general, hence Theorem \ref{main thm} is necessary to prove that they do not admit any filling by holomorphic discs.

	\bibliographystyle{acm}
	\bibliography{Bibliografiatesi}

\begin{thebibliography}{10}

\bibitem{Alexander1}
{\sc Alexander, H.}
\newblock Gromov's method and {B}ennequin's problem.
\newblock {\em Invent. Math. 125}, 1 (1996), 135--148.

\bibitem{borrelli}
{\sc Borrelli, V.}
\newblock Maslov form and {$J$}-volume of totally real immersions.
\newblock {\em J. Geom. Phys. 25}, 3-4 (1998), 271--290.

\bibitem{bryant}
{\sc Bryant, R.~L.}
\newblock Some examples of special {L}agrangian tori.
\newblock {\em Adv. Theor. Math. Phys. 3}, 1 (1999), 83--90.

\bibitem{cielgold}
{\sc Cieliebak, K., and Goldstein, E.}
\newblock A note on the mean curvature, {M}aslov class and symplectic area of
  {L}agrangian immersions.
\newblock {\em J. Symplectic Geom. 2}, 2 (2004), 261--266.

\bibitem{DuvalSibony}
{\sc Duval, J., and Sibony, N.}
\newblock Hulls and positive closed currents.
\newblock {\em Duke Math. J. 95}, 3 (1998), 621--633.

\bibitem{eliashberg}
{\sc Eliashberg, Y.}
\newblock Filling by holomorphic discs and its applications.
\newblock In {\em Geometry of low-dimensional manifolds, 2 ({D}urham, 1989)},
  vol.~151 of {\em London Math. Soc. Lecture Note Ser.} Cambridge Univ. Press,
  Cambridge, 1990, pp.~45--67.

\bibitem{gromovpseudoholom}
{\sc Gromov, M.}
\newblock Pseudo holomorphic curves in symplectic manifolds.
\newblock {\em Invent. Math. 82}, 2 (1985), 307--347.

\bibitem{bryantproceedings}
{\sc Gu, C., Berger, M., and Bryant, R.}
\newblock {\em Differential Geometry and Differential Equations: Proceedings of
  a Symposium, Held in Shanghai, June 21-July 6, 1985}.
\newblock No.~No. 1255 in Lecture Notes in Mathematics. Springer-Verlag, 1987.

\bibitem{Jost}
{\sc Jost, J.}
\newblock {\em Compact Riemann Surfaces:}.
\newblock Universitext (Berlin. Print). Springer, 1997.

\bibitem{joyce}
{\sc Joyce, D.}
\newblock {\em Compact Manifolds with Special Holonomy}.
\newblock Oxford mathematical monographs. Oxford University Press, 2000.

\bibitem{LeBrun}
{\sc LeBrun, C.}
\newblock Twistors, holomorphic disks, and {R}iemann surfaces with boundary.
\newblock In {\em Perspectives in {R}iemannian geometry}, vol.~40 of {\em CRM
  Proc. Lecture Notes}. Amer. Math. Soc., Providence, RI, 2006, pp.~209--221.

\bibitem{YngIngLee}
{\sc Lee, Y.~I.}
\newblock Lagrangian minimal surfaces in {K}\"ahler-{E}instein surfaces of
  negative scalar curvature.
\newblock {\em Comm. Anal. Geom. 2}, 4 (1994), 579--592.

\bibitem{LP}
{\sc Lotay, J., and Pacini, T.}
\newblock {C}omplexified diffeomorphism groups, totally real submanifolds and
  {K}ähler-{E}instein geometry.
\newblock {\em To appear in Trans. American Mathematical Society\/}.

\bibitem{LP1}
{\sc Lotay, J., and Pacini, T.}
\newblock From {L}agrangian to totally real geometry: coupled flows and
  calibrations.
\newblock {\em To appear in Comm. in Anal. Geom.\/}.

\bibitem{LP2}
{\sc Lotay, J., and Pacini, T.}
\newblock From minimal lagrangian to j-minimal submanifolds: persistence and
  uniqueness.
\newblock {\em arXiv:1704.08226\/}.

\bibitem{mcduffsalamon}
{\sc McDuff, D., and Salamon, D.}
\newblock {\em J-holomorphic Curves and Symplectic Topology}.
\newblock No.~v. 52 in American Mathematical Society colloquium publications.
  American Mathematical Society, 2004.

\bibitem{Zehmisch}
{\sc Zehmisch, K.}
\newblock Analytic filling of totally real tori.
\newblock {\em M\"unster J. Math. 9}, 1 (2016), 207--219.

\end{thebibliography}

\end{document}